\documentclass{amsart}
\usepackage[T1]{fontenc}
\usepackage[utf8]{inputenc}
\usepackage{amssymb,url,upref,xspace,smfenum}
\RequirePackage[%
pdfview=XYZ,
pdfstartview=FitV,
pdffitwindow,
hyperfootnotes=false,
raiselinks=true,
hypertexnames=false,
colorlinks=true,
bookmarks=true,
bookmarksopen=false,
anchorcolor=black,
citecolor=black,
linkcolor=black,
filecolor=black,
urlcolor=black]{hyperref}

\newcommand{\red}[1]{#1}
\let\leq\leqslant\let\geq\geqslant
\let\oldbullet\bullet
\renewcommand{\bullet}{{\scriptscriptstyle\oldbullet}}

\usepackage[scr=rsfs,cal=euler]{mathalfa}

\theoremstyle{plain}
\newtheorem*{Theorem}{Theorem}

\theoremstyle{remark}
\newtheorem*{Remarks}{Remarks}

\renewcommand\to{\mathchoice{\longrightarrow}{\rightarrow}{\rightarrow}{\rightarrow}}

\newcommand{\cD}{\mathscr D}
\newcommand{\OO}{\mathscr O}

\newcommand{\CC}{\mathbb C}
\newcommand{\HHH}{\mathscr H}
\newcommand{\rH}{{\scriptscriptstyle\mathrm{H}}}
\newcommand{\RR}{\mathbb R}
\newcommand{\QQ}{\mathbb Q}
\newcommand{\QH}{\QQ^\rH}
\newcommand{\QHn}{\QQ^\rH_n}
\newcommand{\QHm}{\QQ^\rH_m}
\newcommand{\PP}{\mathbb P}

\newcommand{\mm}{\mathfrak m}
\newcommand{\NN}{\mathbb N}
\newcommand{\ZZ}{\mathbb Z}
\newcommand{\DD}{\mathbb D}
\DeclareMathOperator{\CH}{H}

\DeclareMathOperator{\HH}{\mathcal H}
\DeclareMathOperator{\DR}{DR}
\DeclareMathOperator{\Ext}{Ext}
\DeclareMathOperator{\Gr}{Gr}
\DeclareMathOperator{\Ker}{Ker}
\DeclareMathOperator{\ICC}{IC}
\DeclareMathOperator{\MHM}{MHM}
\newcommand{\IC}{\ICC_X}
\newcommand{\cK}{\mathcal K}
\newcommand{\cQ}{\mathcal Q}

\let\oo 0
\let\xx x
\newcommand\hto{\mathrel{\lhook\joinrel\to}}
\makeatletter
\def\p@enumii{}
\makeatother

\begin{document}
\title{Hodge--Lyubeznik numbers}


\author{Ricardo García López}
\address{Departament de Matemàtiques i Informàtica, Universitat de Barcelona. Gran Vía, 585, E-08007 Barcelona, Spain}
\email[R. García López]{ricardogarcia@ub.edu}

\thanks{The first author received partial support from grant PID2022-137283NB-C22 funded by MCIN/AEI/\allowbreak10.13039/501100011033}

\author{Claude Sabbah}
\address{CMLS, CNRS, École polytechnique, Institut Polytechnique de Paris, 91128 Palaiseau cedex, France}
\email[C. Sabbah]{Claude.Sabbah@polytechnique.edu}




\begin{abstract}
We define a Hodge-theoretical refinement of the Lyubeznik numbers for local rings of complex algebraic varieties. We prove that these numbers are independent of the choices made in their definition and that, for the local ring of an isolated singularity, they can be expressed in terms of the Hodge numbers of the cohomology of the link of the singularity. We give examples of isolated singularities with the same Lyubeznik numbers but different Hodge-Lyubeznik numbers. The present version takes into account the published erratum to the first version of this paper.
\end{abstract}


\maketitle

\section*{ Introduction}

Let $A$ be a local ring for which there exists a surjection $\pi: R \to A$, where $R$ is a regular local ring of dimension $n$ containing a field (for instance, $A$ is the local ring of an algebraic variety at a point). Put $I=\Ker\pi$, let $\mm$ be the maximal ideal of $R$, set $k= R/\mathfrak m$.

In \cite{L}, Lyubeznik proved that the numbers
\[
\lambda_{r,s}:= \dim_k\Ext^r_R(k, \CH_I^{n-s}(R))
\]
are finite and depend only on the ring $A$, they are independent of the presentation $A \cong R/I$ (finiteness was known if $R$ contains a field of positive characteristic). They are known as the \emph{Lyubeznik numbers of $A$}, and have been studied by many authors
(see the survey article \cite{NBWZ} and the references therein).

Let $X$ be complex algebraic variety, $\xx\in X$. Taking an affine chart,
we can assume that for some $n\geqslant 1$
we have a closed embedding
$
i: X \hto \CC^n
$
with $i(\xx)=\oo\in\CC^n$.

Set $A=\mathscr O_{X,x}$ and let $D_n$ denote the ring of differential operators on
$\CC^n$ with polynomial coefficients. The embedding $i$ gives a presentation $A\cong\mathscr O_{\CC^n, \oo}/ I$ and the Lyubeznik number $\lambda_{r,s}(A)$ equals the length of the holonomic $D_n$\nobreakdash-module $H^r_{\mathfrak m}(H^{n-s}_I( \mathscr O_{\CC^n, \oo}))$, where $\mathfrak m$ is the maximal ideal of $\mathscr O_{\CC^n, \oo}$.
Let $\cD_n$ denote the sheaf of algebraic linear differential operators on $\mathbb C^n$. Passing to sheaves, the $D_n$-module $H^r_{\mathfrak m}(H^{n-s}_I( \mathscr O_{\CC^n, \oo}))$ corresponds to a $\cD_n$-module $\HH^r_{\{\oo\}}(\HH^{n-s}_{i(X)}( \mathscr O_{\CC^n}))$ with punctual support, which is part of the datum defining a mixed Hodge module \cite{S1}.
Since the category of mixed Hodge modules with punctual support is equivalent to the category of mixed Hodge structures, we have a Hodge and a weight filtration on $\HH^r_{\{\oo\}}(\HH^{n-s}_{i(X)}( \mathscr O_{\CC^n}))$, and we can consider numerical invariants attached to them, say\footnote{The possibility of considering such invariants is mentioned in a remark after Corollary~2 in~\cite{RSW}.}
\begin{equation} \label{hn}
\lambda^{p,q}_{r,s}(\mathscr O_{X,x}):= \dim_{\CC} \Gr_F^p\Gr^W_{p+q} \HH^r_{\oo}\bigl(\HH^{n-s}_{i(X)}\bigl( \QHn[n](n)\bigr)\bigr).
\end{equation}
Here, $\QHn[n](n)$ is defined below, following Saito (see \cite[(4.5.5)]{S1}). It is a Hodge module with~$\mathscr O_{\CC^n}$ as underlying $\cD_n$-module, and it will follow from its definition that
\[
\lambda_{r,s}(\mathscr O_{X,x})=\sum_{p,q}\lambda^{p,q}_{r,s}(\mathscr O_{X,x}).
\]
So, for local rings of complex algebraic varieties, the numbers $\lambda^{p,q}_{r,s}(\mathscr O_{X,x})$ might be regarded as Hodge-theoretical refinements of the Lyubeznik numbers. We prove that these Hodge numbers are independent of the embedding $X \hto \CC^n$, and we give examples showing that they do give more information than the Lyubeznik numbers. We show that for isolated singularities they can be computed in terms of the Hodge numbers of the cohomology of the link of $X$ at $x$, extending the main result in \cite{GS}.

\subsection*{Notations}

We refer to \cite{HTT} for terminology and notations concerning $\cD$-modules and to \cite{S1} and~\cite{S2} for those concerning mixed Hodge modules, see also the introduction to this theory in~\cite{Sch} and also \cite[Chapter 14]{PS}.\footnote{In \cite{S1},\cite{S2}, \cite{Sch} \emph{right} $\cD$-modules are used, while we will use \emph{left} ones.} Still, we briefly
recall a few notations we will use:
\bgroup
\renewcommand{\theenumi}{\roman{enumi}}
\begin{enumerate}
\item\label{nota:i} All varieties considered in this note are defined over the complex numbers. For $n\geqslant 1$, we denote by $\cD_n$ the sheaf of algebraic linear differential operators on $\CC^n$. Unless otherwise specified, $\cD_n$-modules are always left modules.
A mixed Hodge module on a smooth variety $Y$ is a 4-tuple\enlargethispage{\baselineskip}
\[
\mathcal M=(M, F_{\bullet}, K, W_{\bullet}),
\]
where $M$ is a holonomic left $\cD_Y$-module
with regular singularities, $K$ is a rational perverse sheaf on $Y$ together with an isomorphism $\DR(M) \cong K \otimes_{\mathbb Q} \CC$ (omitted in the notation, but part of the data), $F_{\bullet}$ is a good filtration of $M$ by
$\mathscr O_X$-coherent subsheaves and $W_{\bullet}$ is a finite increasing filtration of $K$, inducing one on $M$.
These data has to verify a number of compatibilities and conditions, see \cite[2.17]{S1} or \cite{Sch}.

\item\label{nota:ii} We use the abbreviations MHS for ``mixed Hodge structure'' and MHM for ``mixed Hodge module''. We denote by $\MHM(X)$ the abelian category of mixed Hodge modules on an algebraic variety $X$ \cite[\S4]{S1}. We~will have to consider MHMs on singular varieties embedded in affine spaces. Recall \cite[(2.17.5)]{S1} that if $X$ is an algebraic variety with a closed immersion $X \hto\nobreak Y$, where $Y$ is smooth, then the category $\MHM(X)$ can be identified with the subcategory $\MHM_X(Y)$ of MHMs on $Y$ with support on $X$. If $X$ is a point, then there is an equivalence of categories between $\MHM_X(Y)$ and the category of mixed Hodge structures.

If $\mathcal A$ is an abelian category, we denote by $D^b(\mathcal A)$ its bounded derived category. The symbol $\sim$ denotes an isomorphism in $D^b(\mathcal A)$.
The category $D^b(\MHM(Y))$ is endowed with a six functor formalism (\cite[\S 4]{S1}).
If $X\subset Y$ is a subvariety, we denote $D_X^b(\MHM(Y))$ the full subcategory of $D^b(\MHM(Y))$ which has as objects those complexes whose cohomology is supported on $X$.

\item\label{nota:iii} Following Saito \cite{S1}, we denote by $\QH_{\textup{pt}}$ the constant Hodge module $(\CC, F, \QQ, W)$ on $\{\textup{pt}\}$ with $\Gr^i_F=Gr^W_i=0$ if $i\neq 0$ . If $\alpha_X: X \to \{\textup{pt}\}$ is the constant map, put $\QH_X=\alpha_X^{\ast}\QH_{\textup{pt}}\in D^b(\MHM(X))$. If $Y$ is smooth of dimension $d_Y$, then $\QH_Y[d_Y]$ is concentrated in degree zero and $\QH_Y[d_Y]_0$ is, by abuse of notation, denoted also
\begin{align*}
\QH_Y[d_Y] = (\OO_Y, F_{\bullet}, \QQ_Y[d_Y], W_{\bullet}),\qquad \text{where}\quad
\begin{cases}
\Gr_F^{k}\OO_Y=\{0\} &\text{if }
k\neq 0, \\
\Gr^W_k\OO_Y=\{0\} &\text{if }
k\neq d_Y.
\end{cases}
\end{align*}
For $Y=\CC^n$, we put $\QHn:= \QH_Y$.

\item\label{nota:iv} If $\mathcal M=( M,F_{\bullet}, K, W_{\bullet})$ is a mixed Hodge module and $k \in\mathbb Z$,
then its Tate twist by $k$ is the MHM
\[
\mathcal M(k):=(M, F_{\bullet-k}, (2\pi i) ^k K, W_{\bullet-2k}),
\]
see \cite[2.17]{S1}.

\item\label{nota:v} We will consider the duality functor $\DD$ in the category of mixed Hodge modules as defined by Saito in \cite[2.6]{S1}, see also \cite[2.4.3]{S0}. On the underlying regular holonomic $\cD$-modules, the duality $\DD$ is the usual holonomic duality. For $\QHn[n](n)$ we have
\begin{equation} \label{equa}
\mathbb D \left(\QHn[n](n)\right) \sim \QHn[n].
\end{equation}

\item\label{nota:vi} Let $X$ be an algebraic variety, $Z\stackrel{k}\hto X$ a closed subvariety.
Deligne defines in \cite[(8.3.8)]{Del} a mixed Hodge structure on the (topological) local cohomology $\CH^\ell_Z(X,\QQ)$, $\ell\geqslant 0$. If $\alpha_X : X \to\nobreak \{\textup{pt}\}$, $\alpha_Z: Z \to \{\textup{pt}\}$ are the constant maps then, by a theorem of Saito (see \cite{S2}), we have isomorphisms of MHS
\[
\CH^\ell_Z(X,\QQ)^{\mathrm{Del}} \cong \CH^\ell\bigl(\textup{pt}, (\alpha_X)_{\ast}k_{\ast}k^{!}\alpha_X^{\ast}\QQ\bigr) \cong \CH^\ell\bigl(\textup{pt}, (\alpha_Z)_{\ast}k^{!}\alpha_X^{\ast}\QH_{\textup{pt}}\bigr),
\]
where on the left hand side the MHS is the one of Deligne. For $\mathcal M$ a mixed Hodge module on~$X$, we put
\[
\RR \Gamma_{Z} (\mathcal M) := (\alpha_X)_{\ast}k_{\ast}k^{!}\mathcal M.
\]

\item\label{nota:vii} If $X \hto Y$ is a subvariety of a smooth variety $Y$ then $\mathcal H^{\bullet}_X(\OO_Y)$, the local algebraic cohomology modules of $\OO_{Y}$ supported on $X$, regarded as left $\cD_Y$-modules, are part of the data defining the mixed Hodge modules $\mathcal H^{\bullet}(\RR \Gamma_{X} (\QHn[n](k)))$, for all $k\in\ \mathbb Z$, see \cite{S1}.

\item\label{nota:viii} If $\mathcal M$ is a Hodge module with punctual support (equivalently, a MHS), we put
\[
h^{p,q}(\mathcal M) = \dim_{\CC} \, ( \Gr_F^pGr^W_{p+q} \mathcal M).
\]
\end{enumerate}
\egroup

\section*{Hodge-Lyubeznik numbers}

Let $A$ be a finitely generated $\CC$-algebra, that is, a quotient of a polynomial ring with coefficients in~ $\CC$. Then $A$ is the ring of regular functions of an affine algebraic variety $X$, and any presentation of $A$ as a quotient of $\CC[x_1,\dots, x_n]$ corresponds to a closed embedding $i: X\hto \CC^n$. For a given $x\in X$, we~can assume $i(x)=\oo$. We denote by $\mathscr O_{X,x}$ the local ring of $X$ at $x\in X$.

\begin{Theorem}
Let $X$ be an affine algebraic variety, $\xx\in X$. Choose a closed embedding
$
i: X \hto \CC^n
$
with $n\geqslant 1$, $i(\xx)=\oo\in\CC^n$. Then
\bgroup
\renewcommand{\theenumi}{\alph{enumi}}
\begin{enumerate}
\item\label{enum:a}
The numbers
\begin{equation*}
\lambda^{p,q}_{r,s}(\mathscr O_{X,x}):= \dim_{\CC} \Gr_F^pGr^W_{p+q} \HH^r_{\oo}\bigl(\HH^{n-s}_{i(X)}\bigl( \QHn[n](n)\bigl)\bigr),
\end{equation*}
where $p,q\in\ZZ$, are independent of the embedding $i$.
\item\label{enum:b}
Assume the singularity of $X$ at $\xx$ is isolated and $X$ is of pure dimension $d\geqslant 2$ at $\xx$. Denote by $H_{\{x\}}^i(X, \CC)$ the singular cohomology groups of $X$ with complex coefficients and support on $\{x\}$, endowed with their mixed Hodge structure (see (vi) above).
Then for $p,q\in\ZZ$, one has:
\renewcommand{\theenumii}{\roman{enumii}}
\begin{enumerate}
\item\label{enum:bi}
$\lambda^{p,q}_{0,s}(\mathscr O_{X,x})=h^{-p,-q}(H_{\{x\}}^s(X, \CC))$
for $1 \leq s \leq d-1$,
\item\label{enum:bii}
$\lambda^{p,q}_{r,d}(\mathscr O_{X,x})= h^{p+d, q+d}(H_{\{x\}}^{d+r}(X, \CC))$
for $2 \leq r \leq d$,
\item\label{enum:biii}
all other $\lambda^{p,q}_{r,s}(\mathscr O_{X,x})$ vanish.
\end{enumerate}
\end{enumerate}
\egroup
\end{Theorem}

\begin{proof}
\eqref{enum:a} Following \cite{L}, we first prove a special case: Assume we have an algebraic embedding $\kappa: \CC^n \hto
\CC^m$ with $\kappa(\oo_n)=\oo_m$,
set $j=\kappa \circ i$, we will compare the embeddings given by $i$ and $j$.

We have $\QHn[n](n)\sim\kappa^{!}\,\QH_m[m](m)[m-n]$, so
\begin{align*}
\kappa_{\ast}\,\RR\Gamma_{i(X)}\bigl(\QHn [n](n)\bigr) & \sim \kappa_{\ast}\, i_{\ast}\,i^{!} \,\kappa^{!}\,\QHm [m] (m) [m-n]
\\ &\sim \RR\Gamma_{j(X)}\QHm[m] (m) [m-n],
\end{align*}
Since $\kappa$ is affine, $\kappa_{\ast}$ is the right derived functor of the cohomological
functor $\HHH^0\kappa_{\ast}$, which is exact because $\kappa$ is a closed embedding.
It follows that
\begin{align*}
\HH^{m-r}\bigl(\RR\Gamma_{j(X)}\bigl(\QHm [m](m)\bigr)\bigr) & \cong \HH^{m-r}\bigl(\kappa_{\ast}\,\RR\Gamma_{i(X)}\bigl(\QHn [n](n)\bigr)[n-m]\bigr) \\
& \cong \kappa_{\ast} \HH^{n-r}\bigl(\RR\Gamma_{i(X)}\bigl(\QHn [n](n)\bigr)\bigr),
\end{align*}
see also \cite[Remark 2.5]{MP}.

Let $a:\{\oo\} \hto \CC^n$ and $b=\kappa\circ a$.
Then
\begin{align*}
\RR\Gamma_{\{\oo\}}\Bigl(\HH^{m-r}\bigl(\RR\Gamma_{j(X)}\bigl(\QHm [m](m)\bigr)\bigr)\Bigr) & \sim
b_{\ast}\,b^!\Bigl(\HH^{m-r}\bigl(\RR\Gamma_{j(X)}\bigl(\QHm [m](m)\bigr)\bigr)\Bigr) \\
&\sim \kappa_{\ast}\,a_{\ast}\,a^!\,\kappa^!\,\kappa_{\ast} \HH^{n-r}\bigl(\RR\Gamma_{i(X)}\bigl(\QHn [n](n)\bigr)\bigr) \\
&\sim \kappa_{\ast}\,\RR\Gamma_{\{\oo\}}\bigl(\HH^{n-r}\bigl(\RR\Gamma_{i(X)}(\QHn [n](n)\bigr)\bigr),
\end{align*}
since $\kappa^{!}\kappa_{\ast}\sim \mathrm{id}$. Thus, for any $e\geqslant 0$,
\begin{align*} \label{equality}
\HH^e\Bigl(\RR\Gamma_{\{\oo\}}\bigl(\HH^{m-r}\bigl(\RR\Gamma_{j(X)}\bigl(\QHm [m](m)\bigr)\bigr)\bigr)\Bigr)
& \cong \HH^e\Bigl(\kappa_{\ast}\RR\Gamma_{\{\oo\}}\bigl(\HH^{n-r}\bigl(\RR\Gamma_{i(X)}\bigl(\QHn [n](n)\bigr)\bigr)\bigr)\Bigr) \\
& \cong \kappa_{\ast}\HH^e\Bigl(\RR\Gamma_{\{\oo\}}\bigl(\HH^{n-r}\bigl(\RR\Gamma_{i(X)}\bigl(\QHn [n](n)\bigr)\bigr)\bigr)\Bigr).
\end{align*}

If $\mathcal M$ is a Hodge module on $\mathbb C^n$ with punctual support in $\{\oo\}$
then, since $\MHM_{\{\oo\}}(\CC^n)\simeq\MHM(\{\oo\})\simeq\MHM_{\{\oo\}}(\CC^m)$,
\[
h^{p,q}(\mathcal M) = h^{p,q}(\kappa_{\ast}\mathcal M),
\]
and so the desired equality is proved in this special case. In general, for any $\ell \leqslant N\in\NN$ denote by $a_{\ell,N}: \CC^\ell \hto
\CC^{N}$ the linear embedding defined by $(x_1\dots,x_\ell) \to (x_1, \dots, x_{\ell}, 0,\dots, 0)$. Given embeddings $i: X \hto \CC^n$ and
$j: X \hto \CC^m$, by \cite[Theorem 2]{Sr}
there is an $N> \max\{n,m\}$ and a polynomial automorphism $\varphi: \CC^N \to \CC^N$ such that
$a_{n,N}\circ i= \varphi\circ a_{m,N}\circ j : X \hto \CC^N$. Then, applying the previous argument twice, the result follows.\footnote{It might be that to invoke \cite{Sr} is overkill, it is likely that one could also work with polydisks instead of affine spaces, etc. We prefer to stay within an algebraic framework.}

\eqref{enum:b} The proof is a translation into the MHM language of the proof given in~\cite{GS}. Some arguments that produce numerical equalities in loc.\ cit.\ have to be promoted to isomorphisms of MHM.

First, it is proved in \cite[pg. 321]{GS} that if we consider an embedding $i: X\hto \CC^n$, with $i(\xx)=\oo$, one has an isomorphism of $\cD_n$-modules
\[
\HH_{\{\oo\}}^p\left(\left(\HH_{X}^{-q}\left(\OO_{\CC^n}\right)\right)^\ast\right)\cong\HH^{p+q} \Bigl(\RR \Gamma_{\{\oo\}} \bigl(\RR \Gamma_{X}(\OO_{\CC^n})\bigr)^\ast\Bigr),
\]
where $\ast$ denotes duality for holonomic $\cD$-modules. The same argument as in loc.\ cit.\ gives an isomorphism in the category of mixed Hodge modules, i.e., we have
\begin{equation} \label{equ}
\HH_{\{\oo\}}^r\left(\DD \left(\HH_{X}^{-s}\left( \QHn[n] \right)\right)\right)\cong\HH^{r+s}\left(\mathrm{R}\Gamma_{\{\oo\}}\left( \DD \left(\mathrm{R} \Gamma_{X}\left( \QHn[n]\right)\right)\right)\right).
\end{equation}
If we denote $a:\{\oo\} \hto \CC^n$, $i: X \hto \CC^n$, $k: \{x\} \hto X$ the inclusion maps, then
\begin{align*}
\RR\Gamma_{\{0\}} \DD \left( \RR\Gamma_{X}(\QHn[n])\right)& \sim \RR\Gamma_{\{0\}} \DD i_{\ast}i^{!}\QHn[n] && \text{by definition of $\RR\Gamma_X$} \\
& \sim \RR\Gamma_{\{0\}} i_{\ast} \DD i^{!}\QHn[n] && \text{by properness of $i$} \\
& \sim \RR\Gamma_{\{0\}} i_{\ast} i^{\ast}\QHn[n](n) && \text{by $\DD i^{!}=i^{\ast}\DD$ and (\ref{equa})}\\
& \sim a_{\ast}a^{!} i_{\ast} i^{\ast}\QHn[n](n) && \text{by definition of $\RR\Gamma_{\{0\}}$} \\
& \sim a_{\ast}a^{!}i_{\ast}\QH_X[n](n) && \text{by definition of $\QH_X$}\\
& \sim a_{\ast}k^{!} \QH_X[n](n)&& \text{by \cite[(4.4.3)]{S1}}.
\end{align*}
It follows from these quasi-isomorphisms and (\ref{equ}) that
\begin{equation}\label{eq:hpqHDH}
\begin{aligned}
h^{p,q}\left(\HH_{\{\oo\}}^r\left(\DD \left(\HH_{X}^{-s}\left( \QHn[n] \right)\right)\right) \right) & = h^{p,q} \left(\HH^{r+s} \left(\RR\Gamma_{\{0\}} \DD \left( \RR\Gamma_{X}(\QHn[n])\right) \right)\right) \\
&= h^{p,q} \left( \HH^{r+s} \left(a_{\ast}k^{!} \QH_X[n](n) \right) \right) \\
&= h^{p,q} \left( \HH^{r+s} \left(k^{!} \QH_X[n](n) \right) \right) \\
&= h^{p+n,q+n}\bigl( \HH^{r+s+n}(k^{!} \QH_X)\bigr).
\end{aligned}
\end{equation}
By \eqref{nota:vi} above, we have $\HH^{r+s+n}(k^{!} \QH_X) = H_{\{x\}}^{r+s+n}(X,\CC)$, and so
\begin{equation}\label{eq:HDH}
h^{p,q}\Bigl(\HH_{\{\oo\}}^r\left(\DD \left(\HH_{X}^{-s}\left( \QHn[n] \right)\right)\right) \Bigr) = h^{p+n,q+n}(H_{\{x\}}^{r+s+n}(X,\CC)).
\end{equation}
Since the singularity of $X$ at $x$ is isolated, if $r=0$ and $s=n-i$ with $i<d$, we have
\begin{align*}
h^{p,q}\Bigl(\HH_{\{\oo\}}^{0}\left(\DD \left(\HH_{X}^{-s}\left( \QHn[n]\right)\right)\right) \Bigr)
& = h^{p,q}\left(\DD \left(\HH_{X}^{-s}\left( \QHn[n] \right)\right)\right) \\
& = h^{-p,-q}\left(\HH_{X}^{-s}(\QHn[n])\right) \\
& = h^{-p,-q}\left(\HH^0_{\{ \oo\}}\left(\HH_{X}^{-s}(\QHn[n])\right)\right) \\
& = h^{-n-p,-n-q}\left(\HH^0_{\{ \oo\}}\left(\HH_{X}^{-s}(\QHn[n](n))\right)\right),
\end{align*}
and so,
\begin{align*}
\lambda^{p,q}_{0,i}(\mathscr O_{X,x}) & = h^{p, q}\Bigl(\HH^0_{\{ \oo\}}\bigl(\HH_{X}^{n-i}(\QHn[n](n))\bigr)\Bigr) \\
& = h^{-n-p, -n-q}\Bigl(\HH_{\{\oo\}}^{0}\bigl(\DD \bigl(\HH_{X}^{n-i}( \QHn[n])\bigr)\bigr) \Bigr) =
h^{-p,-q}( H^{i}_{\{x\}}(X, \CC))
\end{align*}
for $i<d$.

For the proof of \eqref{enum:b}\eqref{enum:bii}, we follow \cite{GS} again: Let $\mathrm{Sing}(X)$ denote the singular locus of $X$; since the singularity at $x$ is isolated, there exists a hypersurface~$H$ of $\CC^n$ containing $\mathrm{Sing}(X)\smallsetminus\{x\}$ and such that $x\not\in H$. Replacing~$X$ by $X\smallsetminus H\cap X$, we can assume that $U=X\smallsetminus\{x\}$ is smooth.

In \cite[4.5]{S1}, Saito introduces a Hodge module $\ICC^{\mathrm{Sai}}_X$,\footnote{In \cite{S1}, Saito considers right $\cD$-modules, we consider its left $\cD$-module counterpart.} which is the only object of $\MHM(\CC^n)$ such that its restriction to $U$ is $\QH_U[d]$ and which has no subobject and no quotient object in $\MHM(\CC^n)$ supported at $\oo\in\CC^n$. Its underlying holonomic $\cD_n$-module is the one introduced in \cite[Proposition 8.5 and its proof]{BK}, denoted there by $\mathfrak L (X,\CC^n)$.
Put
\begin{align}\label{eq:K}
\cK & = \HH^0_{\{\oo\}}\bigl(\DD\bigl(\HH^{n-d}_X (\QHn[n])\bigr)\bigr), \\
\IC &= \mathrm{coker} \bigl[\cK\to\DD\bigl(\HH_X^{n-d}(\QHn[n])\bigr) \bigr].\notag
\end{align}
The Hodge module $\IC$ has also no subobjet and (by duality, since $\lambda_{0,d}=0$, see \cite{GS}) no quotient object supported at $\oo$. Its restriction to $U$ is $\QH_U[d](\red{n})$. By uniqueness of Saito's Hodge module, we have $\IC= \ICC^{\mathrm{Sai}}_{X}(n)$.

Since $\cK $ is supported at $\{0\}$, the exact sequence
\[
0 \to \cK \to \DD\bigl(\HH_X^{n-d}(\QHn[n])\bigr) \to \IC \to 0
\]
yields, for $r\geqslant 1$, isomorphisms of MHMs
\[
\HH^r_{\{0\}}\bigl(\DD\bigl(\HH^{n-d}_X (\QHn[n])\bigr)\bigr) \cong \HH^r_{\{0\}}(\IC).
\]

By \cite[4.5.13]{S1}, we have $\DD \IC \cong \IC (d-2n)$, so we also have an exact sequence of MHMs
\begin{equation}\label{eq:exseq}
0 \to \IC (d-2n) \to \HH_X^{n-d}(\QHn[n]) \to \DD \,\cK \to 0,
\end{equation}
and so, for $r\geqslant 2$,
\[
\HH^r_{\{0\}} \IC (d-2n) \cong \HH^r_{\{0\}}\bigl(\HH_X^{n-d}(\QHn[n])\bigr).
\]
It follows that, for $r\geqslant 2$, setting $s=d-n$ in \eqref{eq:hpqHDH},
\begin{align*}
\lambda^{p,q}_{r,d}(\OO_{X,x})& = h^{p,q} \Bigl(\HH^r_{\{0\}}\bigl(\HH_X^{n-d}\bigl(\QHn[n](n)\bigr)\bigr)\Bigr) = h^{p,q} \bigl(\HH^r_{\{0\}} (\IC) (d-n)\bigr) \\
& = h^{p,q}\Bigl(\HH^r_{\{0\}}\bigl(\DD\bigl(\HH^{n-d}_X (\QHn[n])\bigr)\bigr)(d-n)\Bigr) \\
& = h^{p+d-n,q+d-n}\Bigl(\HH^r_{\{0\}}\bigl(\DD\bigl(\HH^{n-d}_X (\QHn[n])\bigr)\bigr)\Bigr) \\
&= h^{p+d,q+d}\bigl(H_{\{x\}}^{d+r}(X,\CC)\bigr).\qedhere
\end{align*}
\end{proof}

\begin{Remarks}\mbox{}
\begin{enumerate}
\item The spectral sequence
\begin{equation*}
E_2^{i,j}=\HH_{\{0\}}^i ( \HH_X^{j}(\OO_{\CC^n})) \Longrightarrow \HH_{\{0\}}^{i+j}(\OO_{\CC^n})
\end{equation*}
is a spectral sequence of MHMs with punctual support. Considering the
Hodge-Euler characteristics
\[
\chi^{p,q}(E_r)= \sum_{i,j} (-1)^{i+j}h^{p,q}(E^{i,j}_r),
\]
and since $\chi^{p,q}(E_r)= \chi^{p,q}(E_{r+1})$, we get
\begin{align*}
\sum_{i,j}(-1)^{i+j}\lambda^{p,q}_{i,n-j}(\OO_{X,x}) & = 0 \quad \text{ if }\ (p,q)\neq (0,0), \\
\sum_{i,j}(-1)^{i+j}\lambda^{0,0}_{i,n-j}(\OO_{X,x}) & = 1\,.
\end{align*}

\item When the singularity of $X$ at $x$ is isolated, the link is a compact manifold of real dimension $2d-1$, and its cohomology is endowed with a natural MHS. From Poincaré duality for the cohomology of the link it follows that, in this case,
\[
\lambda^{p,q}_{0,i} (\OO_{X,x}) = \lambda^{p,q}_{d-i+1, d} (\OO_{X,x})
\]
for $2\leqslant i\leqslant d-1$, $p,q\in\ZZ$.

\item Let $Y \hto \PP^N$ be a projective scheme such that $^{p\!}\HH^j(\QQ_Y)=0$ for $j\neq d$ where $d\in \NN\smallsetminus\{0\}$ and $^{p\!}\HH$ denotes perverse cohomology (for example, this holds if $Y^{an}$ is a $\QQ$-homology manifold of pure dimension). If $C$ is the affine cone over $Y$, it is not difficult to transpose the proof of \cite[1.6--1.8]{RSW} to our setting,\footnote{For the Thom-Gysin sequence, this is remarked at the end of Section 1.3 in \cite{RSW}.} and so it follows that the Hodge-Lyubeznik numbers $ \lambda^{p,q}_{i,j} (\OO_{C,0})$ corresponding to the local ring of $C$ at its vertex are independent of the embedding of $Y$ in projective space.

\item If $k$ is a field of characteristic $p>0$ and $X$ is a $k$-variety with an isolated singularity at $x\in X$, Blickle and Bondu proved in \cite[Theorem 1.1]{BB} that the Lyubeznik numbers of the local ring $\mathscr O_{X,x}$ can be computed in terms of the dimensions over $\ZZ/p\ZZ$ of the étale cohomology of $X$ supported at $x$ with coefficients in $\ZZ/p\ZZ$. It would be interesting to find some refinement of the Lyubeznik numbers also in this setting.
\end{enumerate}
\end{Remarks}

\section*{Examples}

\begin{enumerate}
\item\label{ex:1}
Let $M\subset\PP^{n-1}$ be a connected smooth projective variety of dimension $d-1\geq1$ and let $X\subset\CC^n$ be the cone over $M$. The exact sequence of (topological) local cohomology
\[
\cdots\to H^{s-1}(X) \to H^{s-1}(X\smallsetminus \{\oo\}) \to H^s_{\{\oo\}}(X)\to H^s(X) \to\cdots
\]
lifts as a sequence of MHS. Since $X$ is contractible, it induces an isomorphism of mixed Hodge structures $H^{s-1}(X\smallsetminus \{\oo\}) \cong H^s_{\{\oo\}}(X)$ if $s\geq2$, and $H^s_{\{\oo\}}(X)=0$ if $s\leq1$ since $M$ is connected. Furthermore, denoting by $L$ an ample class, the Thom-Gysin sequence in \cite[Proposition 1.3]{RSW}, which is a sequence of mixed Hodge structures, yields for any $s\geq2$ an exact sequence
\[
H^{s-3}(M)\red{(-1)} \xrightarrow{~L~} H^{s-1}(M)\to H^s_{\{\oo\}}(X) \to H^{s-2}(M)\red{(-1)} \xrightarrow{~L~} H^s(M).
\]
By the hard Lefschetz theorem, we identify $H^s_{\{\oo\}}(X)$ with
\begin{enumerate}
\item\label{ex:1a}
the $L$\nobreakdash-primitive part $H^{s-1}(M)_\mathrm{prim}$, which is pure of weight $\red{s-1}$, if~$s\leq d$ (and $s\geq2$),
\item
and, if $s\geq d+1$, i.e., $s=d+r$ with $r\geq1$, the \red{Tate twisted} $L$\nobreakdash-coprimitive part $H^{s-2}(M)_\mathrm{coprim}\red{(-1)}=\ker L:H^{s-2}(M)\red{(-1)}\to H^s(M)$, which is pure of weight $\red{s}$.
\end{enumerate}
The main theorem yields then 
\bgroup
\renewcommand{\theenumii}{\roman{enumii}}
\begin{enumerate}
\item\label{enum:ci}
$\lambda^{p,q}_{0,s}(\mathscr O_{X,0})=h^{\red{-p,-q}}(H^{s-1}(M)_\mathrm{prim})$ for $2\leq s \leq d-1$,
\item\label{enum:cii}
$\lambda^{p,q}_{r,d}(\mathscr O_{X,0})= h^{p+d\red{-1}, q+d\red{-1}}(H^{d+r-2}(M)_\mathrm{coprim})$
for $2 \leq r \leq d$,
\item\label{enum:ciii}
all other $\lambda^{p,q}_{r,s}(\mathscr O_{X,0})$ vanish.
\end{enumerate}
\egroup

\item
It is known that there exist singularities whose links have equal Betti numbers but different Hodge numbers (see e.g.\ \cite[\S 3]{SS}). In the case of an isolated singularity, Lyubeznik numbers are not enough to recover all Betti numbers of the link, therefore
not all Hodge numbers of the link can be obtained as Hodge-Lyubeznik numbers.
Nevertheless, it is not difficult to give examples of singularities with equal Lyubeznik numbers and distinct Hodge-Lyubeznik numbers:

In \cite{WYW}, the authors find a pair of three dimensional smooth projective complete intersections~$N_1$ and $N_2$ in $\mathbb P^9$ which are diffeomorphic but $h^{0,3}(N_1)\neq h^{0,3}(N_2)$ and $h^{1,2}(N_1)\neq h^{1,2}(N_2)$.\footnote{In fact, they find several pairs with this property.}

For $i=1,2$, let $M_i \subset \PP^{19}$ be the image of $N_i\times \PP^1$ by the Segre embedding, let $X_i \subset \mathbb A^{20}$ be the affine cone over $M_i$, and let $A_i$ denote the local ring at the origin of $X_i$.

Since, for $i=1,2$, the singularity of $X_i$ at zero is isolated, the Lyubeznik numbers of~$A_i$ depend only on the Betti numbers of $X_i\smallsetminus\{0\}$ (see \cite{GS}). Topologically, $X_i\smallsetminus\{0\}$ is a $\mathbb S^1$\nobreakdash-bundle over $M_i \cong N_i\times \PP^1$. It follows that $\lambda_{p,\ell}(A_1)=\lambda_{p,\ell}(A_2)$ for all $p, \ell\in\NN$.

On the other hand,  we apply the results of Example \eqref{ex:1} with $d=5$ and $s=d-1=4$. Thus
\[
\lambda_{0,4}^{-p,-q}(A_i)=h^{\red{p,q}}(H^3(M_i)_\mathrm{prim})=h^{\red{p,q}}(H^3(M_i))-h^{\red{p-1,q-1}}(H^1(M_i)).
\]
But we have $h^{p,q}(H^\ell(M_i))= h^{p,q}(H^\ell(N_i \times \mathbb P^1)) = h^{p,q}(H^\ell(N_i)) + h^{p-1,q-1}(H^{\ell-2}(N_i))$,
and so, $\lambda_{0,4}^{-p,-q}(A_i)=h^{\red{p,q}}(H^3(N_i))$. In particular, if $N_1=X_3 (70, 16, 16, 14, 7, 6)$ we obtain, according to \cite{WYW},
\[
\lambda^{\red{-1,-2}}_{0,4}(A_1)= h^{1,2}(H^3(N_1))=3365330286081,
\]
while if $N_2=X_3 (56, 49, 8, 6, 5, 4, 4)$,
\[
\lambda^{\red{-1,-2}}_{0,4}(A_2)= h^{1,2}(H^3(N_2))=3343868254721.
\]
\end{enumerate}

\pagebreak[2]
\begin{Remarks}\mbox{}
\renewcommand{\theenumi}{\alph{enumi}}
\begin{enumerate}
\item\label{rem:a}
In \cite[Theorem 6.4]{H-P21}, which takes place in the setting of Example \ref{ex:1} above, the authors show that the $\cD_{\CC^n}$-module $\HH_X^{n-d}(\OO_{\CC^n})$ has a maximal simple submodule with support~$X$ (this is the minimal extension of its restriction to $\CC^n\smallsetminus\{0\}$), and that the quotient is supported at the origin with multiplicity equal to $\dim H^{d-1}(M)_\mathrm{prim}$. We can recover this result, and upgrade it to an equality between Hodge numbers as follows. The first point comes from the exact sequence~\eqref{eq:exseq}. For the second point, we write the mixed Hodge module $\cQ$ corresponding to this quotient as $\cQ\simeq\DD\cK$, with $\cK$ defined by \eqref{eq:K}. Then \eqref{eq:HDH} with $s=d-n$ and $r=0$, together with the previous example \ref{ex:1}\eqref{ex:1a} with $s=d$, yields
\[
h^{p,q}\cK(-n) = h^{p,q}(H_{\{0\}}^d(X))=h^{\red{p,q}}(H^{d-1}(M)_\mathrm{prim}),
\]
so that $\cK(-n)$ is pure of weight $d-\red{1}$. It follows that $\cQ(n)$ is pure of weight $\red{1}-d$ and $h^{p,q}\cQ(n)=h^{\red{-p,-q}}(H^{d-1}(M)_\mathrm{prim})$.

\item
The calculation above does not make use of any special property of the varieties in \cite{WYW}, the proof can be transposed to other examples of
homeomorphic complex projective manifolds with different Hodge numbers (according to \cite{WYW} there are examples obtained by Xiao and Campana in the eighties, but we didn't get access to their papers).

\item
It is proved in \cite{Zh} that if $X$ is a projective variety over a field of positive characteristic, the Lyubeznik numbers of the local ring at the vertex of the affine cone over $X$ are independent of the embedding $X \hto \PP^n$.
In contrast, the examples in \cite{RSW} (see also irreducible examples in \cite{Wa}) show that this does not hold for varieties over the complex numbers. It~would be interesting to determine, in these examples, which Hodge-Lyubeznik numbers are involved in the dependence on the embedding $X \hto \PP^n$.
\end{enumerate}
\end{Remarks}

\subsection*{Acknowledgements}
We thank the referee for pointing us the reference \cite{H-P21} and suggesting Remark \eqref{rem:a} above. We thank Bradley Dirks for noticing a missing Tate twist in the exact sequence of Example \ref{ex:1} of the first version of this paper and a typo.



\end{document}